\documentclass{article}
\usepackage{amsmath,amsfonts,amssymb,amsthm,bbm,latexsym,mathrsfs}
\usepackage{graphicx,color,epsfig,fancyhdr,dsfont}
\usepackage{enumerate}
\usepackage{hyperref}
\usepackage{indentfirst}
\usepackage{amsmath,amscd}
\usepackage{caption}
\usepackage{graphicx, subfig}
\usepackage[all]{xy}

\newtheorem{definition}{Definition}
\newtheorem{lemma}[definition]{Lemma}
\newtheorem{theorem}[definition]{Theorem}
\newtheorem{coro}[definition]{Corollary}
\newtheorem{prop}[definition]{Proposition}
\newtheorem{example}[definition]{Example}
\newtheorem{remark}[definition]{Remark}

\newcommand{\Ad}{\mathrm{Ad}}

\newcommand{\GL}{\mathrm{GL}}

\newcommand{\Q}{\mathbb{Q}}

\newcommand{\N}{\mathbb{N}}
\newcommand{\M}{\mathbb{M}}

\newcommand{\Z}{\mathbb{Z}}

\newcommand{\Aut}{\mathrm{Aut}}

\newcommand{\gG}{\mathfrak{g}}

\newcommand{\rR}{\mathfrak{r}}
\newcommand{\sS}{\mathfrak{s}}
\newcommand{\ad}{\mathrm{ad}}
\newcommand{\Hh}{\mathbf{H}}

\newcommand{\Ss}{\mathbf{S}}
\newcommand{\Rr}{\mathbf{R}}
\newcommand{\Lie}{\mathrm{Lie}}
\newcommand{\Sym}{\mathrm{Sym}}
\newcommand{\Tt}{\mathbf{T}}
\newcommand{\Uu}{\mathbf{U}}

\begin{document}
\title{Weak regularity and one-parameter subgroups of definable $p$-adic Lie groups}
\author{Zhentao Zhang}
\date{}
\maketitle

\begin{abstract}
We prove that every group $G$ definable in the pure $p$-adic field $\Q_p$ is weakly regular. We show that every one-parameter subgroup is definable.  We also show that its one-parameter core $G_u$, together with a suitable regular open subgroup $G_\Omega$, is definable. Finally, we show that if $H$ is a dfg component of $G$, then
$G_u=\langle(H^g)_u:g\in G\rangle
$. In fact, $G_u$ is a finite product of the one-parameter cores of finitely many conjugates of $H$. In particular, when $G$ is definably amenable, $G_u=H_u$.
\end{abstract}

\section{Introduction}

Throughout the note, “definable” means definable, with parameters, in the pure field $\Q_p$ in the language of rings. Every definable group is equipped with its canonical definable $p$-adic Lie group structure (\cite{P} Lemma 3.8). For model-theoretic terminology for definable groups, such as definably amenable groups and dfg components, we refer the reader to \cite{PYZ}.

A \textbf{one-parameter morphism} of a $p$-adic Lie group $G$ is a continuous
homomorphism
$$\phi:(\Q_p,+)\longrightarrow G.$$
The subgroup $\phi(\Q_p)$, the image of $\phi$, is called a \textbf{one-parameter subgroup}.

In \cite{paper}, Benoist and Quint   show that  if $G$ is weakly regular, then the \textbf{one-parameter core} $G_u$ which is generated by all
its one-parameter subgroups, is closed and admits a Levi decomposition $G_u= R_u\rtimes S_u$ where $R_u$ is the largest normal algebraic unipotent subgroup of $G$ and $S_u$ is a finite central extension of the Kneser–Tits subgroup. They also construct, for a sufficiently small standard compact open subgroup $\Omega$, an open regular subgroup $G_\Omega=\Omega G_u$ containing all one-parameter subgroups (\cite{paper} Proposition 5.5 and Theorem 5.12).

They also show that $G_u$ is ``almost'' an algebraic  group. Strictly speaking, we work with algebraic groups over a sufficiently large algebraically closed field. For example, when we refer to $R_u$
above as an algebraic unipotent  group, we mean that there exists a unipotent algebraic group 
$\Rr_0$ such that $R_u=\Rr_0(\Q_p)$ (\cite{paper} Lemma 3.2).  Let $\Hh_0$ be an algebraic group which is a semidirect product of  $\Rr_0$ and $\Ss_0$, the Zariski closure of $\Ad(S_u)$. Then the map $\Psi_G: G_u=R_u\rtimes S_u\rightarrow \Hh_0(\Q_p): (r,s)\mapsto (r, \Ad(s))$ is a morphism with finite kernel and finite-index image (\cite{paper} Corollary 5.11).

In this paper, we first prove that every definable group $G$ is weakly regular and naive. Second, we show that every one-parameter subgroup of $G$ is definable and has dfg. Third, we show that $G_u$ is definable and so is some suitable regular open subgroup $G_\Omega$. Finally, we explain $G_u$ in terms of the decomposition theory developed in \cite{PYZ}.

The main results are the following.

\begin{theorem}\label{main-1}
Let $G$ be a group definable in $\Q_p$. Then $G$ is weakly regular and naive.
\end{theorem}

\begin{theorem}\label{main-2}
Let $G$ be a group definable in $\Q_p$. Then
\begin{enumerate}
  \item Every one-parameter subgroup of $G$ is definable and has dfg.
  \item The one-parameter core $G_u$, together with its Levi decomposition $G_u=R_u\rtimes S_u$, is definable.
  \item A sufficiently small standard compact open subgroup $\Omega$ can be chosen to be definable, and then $G_\Omega=\Omega G_u$ is a definable regular open subgroup.
\end{enumerate}
\end{theorem}

\begin{theorem}\label{main-3}
Let $G$ be a group definable in $\Q_p$. Let $H$ be a dfg component of $G$ in $\Q_p$. Then 
$$G_u=\langle H^g_u:g\in G\rangle.$$
Moreover, there are $g_1,\dots, g_n\in G$ such that
$$G_u= H^{g_1}_u\cdots H^{g_n}_u.$$
In particular, when $G$
is definably amenable, $G_u=H_u$.
\end{theorem}

\section{Weak regularity}

We recall the terminology of \cite{paper}.

\begin{definition}
A $p$-adic Lie group $G$ is \textbf{weakly regular} if any two one-parameter morphisms with the same derivative are equal. It is \textbf{$\Ad$-regular} if $\ker (\Ad_G)=Z(G)$. It is \textbf{naive} if the finite subgroups of $G$ have uniformly bounded cardinality. It is \textbf{regular} if it is both $\Ad$-regular and naive.
\end{definition}

Every regular $p$-adic Lie group is weakly regular (\cite{paper} Lemma 2.7). The converse need not
hold, even for definable groups:

\begin{example}
Let $G=(\Q_p,+)\rtimes \{\pm 1\}$ where $-1$ acts on $(\Q_p,+)$ by $x\mapsto -x$. Obviously, $G$ is definable.  Its Lie algebra is one-dimensional and abelian. Then for $(x,\epsilon)\in G$, $\Ad_G(x,\epsilon)=\epsilon$ and $\ker(\Ad_G)=\Q_p\times\{1\}$. While $Z(G)=\{1_G\}$. Hence, $G$ is not $\Ad$-regular.

We will show that every definable group is weakly regular and naive, so $G$ is a definable, weakly regular and non-regular group.
\end{example}

Now we show that every definable group is weakly regular and naive.


\begin{lemma}\label{lemma-naive}
Let $1\rightarrow A\rightarrow B\stackrel{\pi}{\rightarrow} C\rightarrow 1$ be an exact sequence. If $A$ and $C$ are naive, then $B$ is naive.
\end{lemma}
\begin{proof}
Assume that the size of finite subgroups of $A$ is bounded by $m$ and the size of finite subgroups of $C$ is bounded by $n$.
Let $F$ be finite. Then $|F|=|F\cap A||\pi(F)|\leq mn$.
\end{proof}

\begin{lemma}\label{lemma-commutative-regular}
Let $G$ be a commutative definable group. Then $G$ is regular, and hence weakly regular.
\end{lemma}
\begin{proof}
 Obviously, $G$ is $\Ad$-regular.  It remains to prove that it is naive.

The group $G$ is definably amenable. By the dfg-fsg decomposition theorem (\cite{PYZ} Theorem 8.8),
there is a definable exact sequence
$$1\rightarrow H\rightarrow G\rightarrow C\rightarrow 1$$
where $H$ is a dfg group and $C$ is an fsg group.

The group $C$ is compact and is therefore naive by \cite{paper} Example 2.6(a).

By the characterization of dfg groups (\cite{PYZ} Theorem 5.6), there are a finite-index definable
subgroup $H'\leq H$, a connected $\Q_p$-split solvable algebraic group $\Hh$, and a definable
homomorphism $f: H'\rightarrow \Hh(\Q_p)$ with  finite kernel and finite-index image. The linear group $\Hh(\Q_p)$ is naive by \cite{paper} Example 2.6(b). Then it is easy to see that $H$ is naive.

By Lemma \ref{lemma-naive}, $G$ is naive.
\end{proof}

\begin{lemma}\label{lemma-comm-by-finite}
Let $G$ be a commutative-by-finite definable group. Then $G$ is weakly regular and naive.
\end{lemma}
\begin{proof}
Let $K$ be a commutative subgroup of finite index. For $g\in K$, the centralizer $Z_G(g)$ is definable and $K\leq Z_G(g)$. Since $K$ has a finite index, there are only finitely many possible $Z_G(g)$ for $g\in K$. Then $Z_G(K)=\bigcap_{g\in K} Z_G(g)$ is, in fact, a finite intersection, and hence definable.  Then $Z_G(Z_G(K))$ is a definable finite-index commutative subgroup of $G$.  Hence, we may assume that $K$ is definable.

By Lemma \ref{lemma-commutative-regular}, $K$ is naive, and it is easy to see that $G$ is naive.

Now let $\phi_1,\phi_2:(\Q_p,+)\rightarrow G$ be one-parameter morphisms with the same derivative. Since $(\Q_p,+)$ is divisible, the images of $\phi_1$ and $\phi_2$ are, in fact, in $K$. As $\Lie(G)=\Lie(K)$, the derivatives of $\phi_1$ and $\phi_2$ remain equal when $\phi_1$ and $\phi_2$ are viewed as maps into $K$. By Lemma \ref{lemma-commutative-regular}, $\phi_1=\phi_2$.
\end{proof}

\begin{prop}
Every definable group is naive.
\end{prop}
\begin{proof}
By \cite{JY} Lemma 2.11, there is a definable exact sequence $1\rightarrow A\rightarrow G\rightarrow L\rightarrow 1$
where $A$ is commutative-by-finite and $L$ is  definably isomorphic to a definable
subgroup of some $\GL_n(\Q_p)$. The group $A$ is naive by Lemma \ref{lemma-comm-by-finite}, and $L$ is naive by \cite{paper}
Example 2.6(b). Then apply Lemma \ref{lemma-naive}.
\end{proof}

\begin{lemma}\label{lemma-Ad-ker}
Let $G$ be a definable group and $\gG$ its Lie algebra. Then the image of every one-parameter morphism $\phi:(\Q_p,+)\rightarrow G$ centralizes $\ker(\Ad_G:G\rightarrow \GL(\gG))$.
\end{lemma}
\begin{proof}
Take $a\in \ker(\Ad_G)$. Since $\Ad(a)=1$, we have $\Lie(Z_G(a))=\gG^{\Ad(a)}:=\{X\in \gG:\Ad(a)(X)=X\}=\gG$. The conjugacy class $a^G:=\{a^g:g\in G\}$ is definable and $0$-dimensional. So $a^G$ is finite. The conjugation
through $\phi$ induces a group morphism $\theta:(\Q_p,+)\rightarrow \Sym(a^G)$ where $\Sym(a^G)$ is the symmetric group of $a^G$. Since $(\Q_p,+)$ is divisible and $\Sym(a^G)$ is finite, the morphism $\theta$ is trivial. Hence, $\phi(\Q_p)$  centralizes  $\ker(\Ad_G)$.
\end{proof}

\begin{prop}
Every definable group is weakly regular.
\end{prop}
\begin{proof}
Let $G$ be a definable group and $\gG$ its Lie algebra. Let $A=\ker(\Ad_G:G\rightarrow\GL(\gG))$.
By  \cite{JY} Theorem 2.9, $A$ is commutative-by-finite. Hence it is weakly regular
by Lemma \ref{lemma-comm-by-finite}.

Let $\phi_1,\phi_2:(\Q_p,+)\rightarrow G$ have the same derivative $X$. The two linear one-parameter morphisms $\Ad\circ \phi_1$ and $\Ad\circ \phi_2$ have the same derivative $\ad(X)$, so they agree by the uniqueness of linear one
parameter morphisms (\cite{paper} Lemma 2.2). Hence,
$\theta(t)=\phi_1^{-1}(t)\phi_2(t)$ lies in $A$.

By Lemma \ref{lemma-Ad-ker}, for $s,t\in \Q_p$, we have $\theta(s+t)=\phi_1^{-1}(s+t)\phi_2(s+t)=\phi_1^{-1}(t)\phi_1^{-1}(s)\phi_2(s)\phi_2(t)=\phi_1^{-1}(t)\theta(s)\phi_1(t)\theta(t)=\theta(s)\theta(t)$. Hence, $\theta:(\Q_p,+)\rightarrow A$ is a one-parameter morphism. Its derivative is zero. The weak regularity of $A$
implies that $\theta$ is trivial, and hence $\phi_1=\phi_2$.
\end{proof}

In summary, we have proved Theorem \ref{main-1}.

\section{Definability of one-parameter subgroups}

We first give two easy consequences of naivety.

\begin{lemma}\label{easy-1}
Let $G$ be a naive $p$-adic Lie group and $\phi:(\Q_p,+)\rightarrow G$  a one-parameter
morphism. Then $\phi$ is either trivial or injective.
\end{lemma}
\begin{proof}
The closed subgroups of $(\Q_p,+)$ are precisely $\{0\}$, $\Q_p$ and  $p^m\Z_p$ for $m\in\Z$. If $\ker(\phi)=p^m\Z_p$, then the quotient $\Q_p/p^m\Z_p$ contains finite cyclic subgroups of order $p^n$ for every $n\in\N$. Their images are finite subgroups of $G$ of unbounded order, which contradicts naivety.
\end{proof}

\begin{lemma}\label{easy-2}
Let $G$ be a compact $p$-adic Lie group and $\phi:(\Q_p,+)\rightarrow G$ a one-parameter
morphism. Then $\phi$ is trivial.
\end{lemma}
\begin{proof}
Let $K$ be an open normal uniform pro-$p$ subgroup of $G$. Since $(\Q_p,+)$ is divisible, $\phi(\Q_p)$ lies in $K$. For $x=\phi(t)$ and $n\geq 1$, choose $s\in\Q_p$ with $p^ns=t$. Then $x=\phi(s)^{p^n}\in K^{p^n}$. As $\bigcap_{n\geq 0}K^{p^n}=\{1_G\}$, $x=1_G$.
\end{proof}

\begin{prop}\label{prop-def-one}
Let $G$ be definable in $\Q_p$ and $\phi:(\Q_p,+)\rightarrow G $ a one-parameter morphism.
Then $U=\phi(\Q_p)$ is a definable subgroup of $G$. If it is nontrivial, it is definably isomorphic to a linear unipotent one-parameter group, and in
particular, $U$ has dfg.
\end{prop}
\begin{proof}
 The trivial case is clear. We assume that $U$ is non-trivial and, by Theorem \ref{main-1} and Lemma \ref{easy-1}, $U\cong (\Q_p,+)$. 

 Let $A=\ker(\Ad_G)$. By Lemma \ref{lemma-Ad-ker},  $U$ centralizes $A$. Let $V=\Ad_G(U)$.

\noindent \underline{Case 1. $V$ is trivial}. Then $U\leq A\cap Z_G(A)=Z(A)$.  So $U$ is contained in the definable commutative group $K=Z(A)$.

\noindent \underline{Case 2. $V$ is non-trivial}. Then $V$ is a nontrivial linear one-parameter group  given by an injective morphism $t\mapsto \exp(tX)$ with $X$ nilpotent (\cite{paper} Lemma 2.2).
Here, $\exp(tX)=\sum_{n\geq 0}\frac{t^nX^n}{n!}$ is, in fact, a finite sum and hence the morphism $t\mapsto \exp(tX)$ is definable.
Let $P=\Ad_G^{-1}(V)$ and $K=Z_P(A)$. Both $P$ and $K$ are definable.  For $x\in K$, choose $u\in U$ with $\Ad_G(x)=\Ad_G(u)$. Then
$xu^{-1} \in A\cap K =Z(A)$. Thus,
$K =Z(A)U$. Since $U$ centralizes $A$, the group $K$ is commutative.

Anyway, in either case, $U$ lies in a definable commutative group $K$. Apply the dfg-fsg decomposition on $K$ (\cite{PYZ} Theorem 8.8):
$$1\rightarrow H\rightarrow K\rightarrow C\rightarrow 1,$$
where $H$ has dfg and $C$ is compact. By Lemma \ref{easy-2}, $U\leq H$. By the description of $H$ (\cite{PYZ} Theorem 5.6), there are a finite-index definable subgroup $H'\leq H$, a connected split solvable algebraic group $\Hh$, and a definable homomorphism $f: H'\rightarrow \Hh(\Q_p)$ with finite kernel and finite-index image. 

The divisibility of $U$ implies $U\leq H'$. The image $L=f(U)$ is a linear one-parameter subgroup, which is of the form $\exp(tX)$ and is hence definable. Let $D=f^{-1}(L)$. Then $D$ is a definable commutative group and $D=U F$ where $F=\ker (f)$.

Let $m$ be an exponent of the finite group $F$. We claim that $$U=D^m:=\{d^m:d\in D\}.$$
The divisibility of $U$ implies $U\leq D^m$. The other direction is easy to check: for $u\in U$ and $a\in F$, $(ua)^m=u^m a^m=u^m\in U$.

Hence, $U$ is definable. Moreover, $U\cap F$ is trivial, since $U$ is torsion-free. Hence, $f|_U:U\rightarrow L$ is a definable isomorphism. 
Then $U$ has dfg because $L$ has dfg.
\end{proof}

\begin{remark} The above proposition does not assert that the one-parameter morphism $\phi$ itself is
definable. It asserts that its image is a definable subgroup.
\end{remark}

\section{Definability of $G_u$}

Let $G$ be a weakly regular $p$-adic Lie group. Following \cite{paper}, write $G_u$ for the subgroup
generated by all one-parameter subgroups. Let $\rR$ be the solvable radical of $\gG= \Lie(G)$ and $\sS$ a
Levi subalgebra, i.e., a maximal
semisimple Lie subalgebra.  The Levi decomposition gives $\gG=\rR\rtimes\sS$.

Let $R_u$ be the subgroup of $G$ generated by
all the one-parameter subgroups of $G$ tangent to $\rR$, and $S_u$
 the subgroup of $G$ generated by all the one-parameter subgroups of $G$ tangent to $\sS$. 
 Benoist and Quint show that
 $$G_u=R_u\rtimes S_u \text{ and } R_u\cap S_u=\{1_G\},$$
where $R_u$ is algebraic unipotent and
$\Ad: S_u \rightarrow S^+:=\Aut (\sS_u)_u$
is onto with finite kernel (\cite{paper} Proposition 5.5). Here, $\rR_u=\Lie(R_u)$ and $\sS_u=\Lie(S_u)$.



Note that, if $G$ is definable, by Theorem \ref{main-1}, $G$ is weakly regular and then $G_u$, $R_u$ and $S_u$ are meaningful.

\begin{prop}\label{def-Ru}
Let $G$ be definable. Then $R_u$ is definable.
\end{prop}
\begin{proof}
Recall that  $R_u$  is isomorphic
to the group of $\Q_p$-points of a unipotent algebraic group and its exponential and logarithm
coordinates are polynomial (\cite{paper} Lemma 3.2). Equivalently, if $X_1,\dots,X_m$ is a basis of $\rR_u$ adapted to a central series, then
$$R_u =\exp(\Q_p X_1)\cdots \exp(\Q_p X_m).$$
Each factor is a one-parameter subgroup of $G$ and is definable by Proposition \ref{prop-def-one}.
Hence, $R_u$ is definable.
\end{proof}

\begin{prop}\label{def-Su}
Let $G$ be definable. Then $S_u$ is definable.
\end{prop}
\begin{proof}
By \cite{paper} Proposition 5.5, the adjoint representation induces a short exact sequence of
topological groups $1\rightarrow F\rightarrow S_u\stackrel {\ell}{\rightarrow} S^+\rightarrow 1$ where $S^+=\Aut(\sS_u)_u$ and $F$ is finite and central.

We write $\sS_u=\sS_1\oplus\dots\oplus \sS_m$ as a direct sum of its simple ideals.
Note that $\sS_u$ is totally isotropic, i.e., is spanned by nilpotent elements. Then each $\sS_j$ admits a basis $X_{j,1},\dots ,X_{j,d_j}$ consisting of nilpotent elements. For each $j,k$, we let 
$$\psi_{j,k}:(\Q_p,+)\rightarrow S^+: t\mapsto \exp(t \text{ } \ad(X_{j,k}))$$
and $V_{j,k}:=\psi_{j,k}(\Q_p)$. 

For each simple ideal $\mathfrak s_j$, let
  $\Ss_j=\Aut(\mathfrak {s}_j)^\circ$, where $\Aut(\sS_j)$ is treated as an algebraic group. 
  Then $\Ss_j$ is a $\Q_p$-isotropic almost $\Q_p$-simple algebraic group. Let $S_j^+=\mathbf S_j(\mathbb Q_p)^+$ be the Kneser–Tits subgroup. 
  Although $\Aut(\sS_u)$ may permute isomorphic simple ideals, every one-parameter subgroup acts trivially on this finite permutation quotient. Hence, the subgroup generated by one-parameter subgroups preserves every simple ideal and we have
  $S^+
  \cong \prod_{j=1}^m S_j^+$. 

 Let $D_j:=\langle V_{j,k}:1\leq k\leq d_j\rangle$. Then $D_j$ is a subgroup of $S_j^+$. Consider the analytic map $m_j: \Q_p^{d_j}\rightarrow S_j^+:(t_1,\dots,t_{d_j})\rightarrow \prod_{k=1}^{d_j}\psi_{j,k}(t_k)$. Then the derivative at zero $(d m_j)_0$ is $\sum_{k=1}^{d_j}t_k \ad(X_{j,k})$. Since $S_j^+$ is open in $\mathbf {S}_j(\mathbb Q_p)$, $\Lie(S_j^+)=\Lie(\Ss_j(\Q_p))$. Since $\Ss_j=\Aut(\sS_j)^\circ$, we have $\Lie(\Ss_j(\Q_p))=\mathrm{Der}(\sS_j)$.  As $\mathfrak s_j$ is semisimple, its center is trivial and every derivation of $\mathfrak s_j$ is inner. Hence,
$\ad:\sS_j \rightarrow\mathrm{Der}(\sS_j)=\Lie(\Ss_j(\Q_p))=\Lie(S_j^+)$
is an isomorphism. Then it is easy to check that the linear map $$(dm_j)_0:\Q_p^{d_j}\rightarrow \Lie (S_j^+):(t_1,\dots, t_{d_j})\mapsto \sum_{k=1}^{d_j}t_k \ad(X_{j,k})$$ is an isomorphism. Hence, by the $p$-adic inverse function theorem, the product $V_{j,1}\cdots V_{j,d_j}$  contains an open neighborhood of the identity in $S_j^+$, and $D_j$ is an open subgroup of $S_j^+$. 

Since $D_j$ contains a nontrivial unipotent one-parameter subgroup, it is unbounded in $S_j^+$. By 
\cite{PYZ} Corollary 3.15, every open subgroup of a $\Q_p$-isotropic almost $\Q_p$-simple $p$-adic algebraic group is either bounded or contains the Kneser–Tits
subgroup. Hence, $D_j=S_j^+$. 

Let $Y=\bigcup_{j=1}^m\bigcup_{k=1}^{d_j}V_{j,k}$.
Obviously, $Y$ is definable. Let $\M$ be a very saturated extension of $\Q_p$. The preceding argument applies over $\M$. Hence, by \cite{PYZ} Corollary 3.15, we have that $S^+$ is definable and  $S^+(\M)=\langle Y(\M)\rangle=\bigcup_{n\geq 1} Y(\M)^n$. By the saturation of $\M$, the union $S^+(\M)=\bigcup_{n\geq 1} Y(\M)^n$ is, in fact, a finite union, so there is $N\geq 1$ such that $S^+(\M)=Y(\M)^N$. Then, by elementarity, we have $S^+=Y^N$ in $\Q_p$. 

We now lift the finitely many one-parameter subgroups. By \cite{paper} Lemma 5.9, every morphism
$\psi_{j,k}$ lifts through the finite central extension $1\rightarrow F\rightarrow S_u\stackrel {\ell}{\rightarrow} S^+\rightarrow 1$. Namely, there is a one-parameter morphism $\phi_{j,k}:(\Q_p,+)\rightarrow S_u$ such that $\ell\circ\phi_{j,k}=\psi_{j,k}$. 
Let $U_{j,k}=\phi_{j,k}(\Q_p)$. By Proposition \ref{prop-def-one}, $U_{j,k}$ is definable. Let $X=\bigcup_{j=1}^m\bigcup_{k=1}^{d_j} U_{j,k}$. Of course, $X$ is definable. 

Since $\ell(X)=Y$ and $S^+=Y^N$, we have $S_u=FX^N$.  Hence, $S_u$ is definable.
\end{proof}

\begin{coro}\label{coro-finite-prod}
Let $G$ be definable. Then $G_u$ is definable. Moreover, $G_u$ is a finite product of one-parameter subgroups.
\end{coro}
\begin{proof}
Note that $G_u=R_u\rtimes S_u$. Since $R_u$, $S_u$ and the group multiplication are definable, $G_u$ is definable. 

Moreover, from the proof of Proposition \ref{def-Ru}, $R_u$ is a product of finitely many one-parameter subgroups. From the proof of Proposition \ref{def-Su}, $S_u$ is a product of finitely many one-parameter subgroups and a finite group $F$. Obviously, the finite group $F\leq S_u\leq G_u$ is contained in a finite product of one-parameter subgroups. Thus, $G_u=R_u\rtimes S_u$ is also a product of finitely many one-parameter subgroups. 
\end{proof}

As a corollary, we give the following result about $\Psi_G$, which is mentioned in the introduction.
Recall that $\Hh_0$ is an algebraic group and $\Hh_0=\Rr_0\rtimes \Ss_0$ where $\Rr_0(\Q_p)=R_u$ and $\Ss_0$ is the Zariski closure of $\Ad(S_u)$. The morphism $\Psi_G:G_u=R_u\rtimes S_u\rightarrow \Hh_0(\Q_p):(r,s)\mapsto (r,\Ad(s))$ has finite kernel and finite-index image. It is immediate from the definability of $G_u$, $R_u$, $S_u$, and the adjoint map $S_u\rightarrow \Ss_0(\Q_p): s\mapsto \Ad(s)$ that

\begin{coro}\label{coro-alge}
The morphism $\Psi_G$ is definable. 
\end{coro}

We now deal with the regular open subgroup $G_\Omega$. Note that every compact open subgroup of a definable group is definable (\cite{OP} Remark 1.5). Hence, the sufficiently small standard open subgroup $\Omega$ can be chosen to be definable. Then we immediately have

\begin{coro}
The group $G_\Omega=\Omega G_u$ is definable.
\end{coro}

In summary, we have proved Theorem \ref{main-2}.

\section{Recovery $G_u$ by dfg components}

We first study dfg groups. Let $H$ be a dfg group. Recall that by  \cite{PYZ}  Theorem 5.6, there are a connected $\Q_p$-split solvable algebraic group $\Hh$, a finite-index definable
subgroup $H'\leq H$,  and a definable
homomorphism $f: H'\rightarrow \Hh(\Q_p)$ with finite kernel and finite-index image. Moreover, $\Hh=\Uu\rtimes\Tt$ where $\Uu$ is the unipotent radical of $\Hh$ and $\Tt$ is a $\Q_p$-split torus.  
Then we have 

\begin{prop}
Every one-parameter subgroup of $H$ lies in $H'$ and the restriction $f|_{H_u}:H_u\rightarrow \Uu(\Q_p)$ is a definable Lie-group isomorphism.
\end{prop}
\begin{proof}
Since $(\Q_p,+)$ is divisible and $H'$ is of finite index in $H$, every one-parameter subgroup of $H$ lies in $H'$ and $H_u=H'_u$. Let $F=\ker (f)$ and $K=Z_{H'}(F)$. Then we have a finite central extension $1\rightarrow F_0\rightarrow K\rightarrow f(K)\rightarrow 1$ where $F_0=F\cap K=Z(F)$. The conjugation action of $H'$ on $F$ induces an embedding  $H'/K\hookrightarrow \Aut(F)$, so $K$ has finite index in $H'$. Then every one-parameter subgroup of $H'$ lies in $K$, so $K_u=H'_u=H_u$. 

Recall that a linear one-parameter group is unipotent (\cite{paper} Lemma 2.2), so $f(H_u)\leq \Uu(\Q_p)$.  Since $f(K)$
has finite index in $\Hh(\Q_p)$, the subgroup $f(K)\cap \Uu(\Q_p)$ has finite index in $\Uu(\Q_p)$. An algebraic
unipotent $p$-adic group has no proper finite-index subgroup, so $\Uu(\Q_p)\leq f(K)$.

Let $\psi:(\Q_p,+)\rightarrow \Uu(\Q_p)$ be a one-parameter morphism. By  \cite{paper} Lemma 5.9, $\psi$ lifts a one-parameter morphism $\phi:(\Q_p,+)\rightarrow K$. Since $\Uu(\Q_p)$ is generated by its one-parameter subgroups, we have $f(H_u)=\Uu(\Q_p)$.

Note that the Lie algebra of $H$ is solvable, so the semisimple part of $H_u$ is trivial and $H_u$ is  algebraic unipotent.  In particular, $H_u$ is torsion-free. Hence, $H_u\cap F$ is trivial and $f|_{H_u}$ is an isomorphism.
\end{proof}

Now we let $G$ be definable and recover $G_u$ by its dfg components. Note that the dfg components mentioned are all assumed to be definable over $\Q_p$ and defined in $\Q_p$. And for every two dfg component $H_1$ and $H_2$, there is $g\in G$ such that $H_1^g\cap H_2$ has finite index in both $H_1^g$ and $H_2$ (\cite{PYZ} Theorem B). Assume that $H$ is a dfg component of $G$. 
It is obvious that, for $g\in G$, $(H_u)^g=(H^g)_u$, so we just write $H_u^g$.
Then the subgroup $\langle H^g_u:g\in G\rangle$ of $G$ is precisely the subgroup generated by all one-parameter cores of dfg components.

\begin{prop}\label{prop-Gu-Hu}
$G_u=\langle H^g_u:g\in G\rangle$.
\end{prop}
\begin{proof}
By \cite{PYZ} Theorem 8.17, the homogeneous space $G/H$ is compact.
Let $U$ be a one-parameter subgroup of $G$. By Proposition \ref{prop-def-one}, $U$ is definable and has dfg.
The action of $U$ on $G/H$ is definable and continuous. By \cite{PYZ} Proposition 6.3, there is a coset $gH$ with a finite $U$-orbit. Its
stabilizer $U\cap H^g$
has finite index in $U$. 

Since $U\cong (\Q_p,+)$ is divisible and has no proper finite-index subgroup, $U \leq H^g$.
Thus every one-parameter subgroup of $G$ lies in a conjugate of $H$, which gives $G_u\leq \langle H_u^g: g\in G\rangle$.

The reverse inclusion is immediate because every one-parameter subgroup of a conjugate of
$H$ is a one-parameter subgroup of $G$.
\end{proof}

Note that, by Corollary \ref{coro-finite-prod}, $G_u$ is a finite product of one-parameter subgroups. Thus, we have

\begin{coro}
There are $g_1,\dots, g_n\in G$ such that $G_u= H^{g_1}_u\cdots H^{g_n}_u$.
\end{coro}
\begin{proof}
Assume that $G_u=U_1\cdots U_n$, a finite product of one-parameter subgroups. Then from the proof of Proposition \ref{prop-Gu-Hu}, for each $1\leq j\leq n$, there is $g_j\in G$ such that $U_j\leq H^{g_j}_u$. Then $G_u= U_1\cdots U_n\subset  H^{g_1}_u\cdots H^{g_n}_u\subset G_u$. Thus, $G_u= H^{g_1}_u\cdots H^{g_n}_u$.
\end{proof}

Note that when $G$ is definably amenable, for every $g\in G$, $H^g\cap H$ has finite index in $H$. Moreover, $H$ can be chosen to be normal (\cite{PYZ} Theorem A). So we immediately have

\begin{coro}
Let $G$ be definably amenable. Then $G_u=H_u$.
\end{coro}

We can also give an algebraic version by the definability of $\Psi_G$ (Corollary \ref{coro-alge}). Let $\Uu_g$ be the Zariski closure of $\Psi_G(H^g_u)$ in $\Hh_0$. Since $H_u^g$ is nilpotent, the image  $\Psi_G(H_u^g)$ is nilpotent, and therefore its Zariski closure $\mathbf U_g$ is nilpotent and hence solvable.
Moreover, $\mathbf U_g$ is connected,
since $\Psi_G(H_u^g)$ is generated by one-parameter subgroups, each of which has connected unipotent Zariski closure. By the structure theorem for connected solvable algebraic groups
(Section 2.1.8 \cite{PR}),
the unipotent elements of $\Uu_g$ form a closed normal subgroup $(\Uu_g)_u$. Since all the one-parameter
generators of $\Psi_G(H_u^g)$ lie in $(\Uu_g)_u$, their Zariski density implies
  $\Uu_g=(\Uu_g)_u$.
Hence, $\Uu_g$ is a unipotent algebraic subgroup of $\Hh_0$. Moreover, by Proposition \ref{prop-Gu-Hu} and the Zariski density of $\Psi_G(G_u)$ in
$\mathbf H_0$, $\Hh_0$ is generated by those $\Uu_g$ in the sense of algebraic groups.

\end{document}